\documentclass[10pt]{amsart}
\usepackage{xypic,somedefs,tracefnt,latexsym,rawfonts,graphicx,amsfonts,amssymb,latexsym,enumerate,amscd}
\def\cal{\mathcal}
\def\Bbb{\mathbb}

\def\r{\rangle}
\def\l{\langle}

\newtheorem{thm}{Theorem}[section]
\newtheorem{exm}{Example}[section]

\newtheorem{lemma}{Lemma}[section]

\newtheorem{defn}{Definition}[section]

\newtheorem{rem}{Remark}[section]
\numberwithin{equation}{section}
\begin{document}
\title[A sufficient condition for a $2$-dimensional orbifold to be good]{A sufficient condition for a
  $2$-dimensional orbifold to be good}
\author[S.K. Roushon]{S. K. Roushon}
\address{School of Mathematics\\
Tata Institute\\
Homi Bhabha Road\\
Mumbai 400005, India}
\email{roushon@math.tifr.res.in} 
\urladdr{http://www.math.tifr.res.in/\~\ roushon/}
\thanks{January 28, 2019}
\thanks{To appear in The Mathematics Student}
\begin{abstract}
  We prove that a connected $2$-dimensional orbifold with finitely
  generated and infinite orbifold fundamental
group is good. We also describe all the good $2$-dimensional orbifolds with finite orbifold fundamental groups.\end{abstract}

\keywords{Orbifold, orbifold fundamental group.}

\subjclass[2010]{Primary: 14H30, 57M60 Secondary: 57M12.}
\maketitle

%\tableofcontents

\section{Introduction}

We start with a short introduction to orbifolds.

The concept of orbifold was first introduced in \cite{Sat}, and 
was called `V-manifold'. Later, it was revived in \cite{Thu}, 
with the new name {\it orbifold}, and 
{\it orbifold fundamental group} of a 
connected orbifold was defined.

\begin{defn} {\rm An {\it orbifold} is a second countable and Hausdorff topological space $M$, 
which at every point looks like the quotient space of ${\Bbb R}^n$, 
for some $n$, by some finite group action. More precisely, there is an open covering 
$\{U_i\}_{i\in I}$ of $M$ together with a collection ${\cal O}_M=\{(\tilde U_i, p_i, G_i)\}_{i\in I}$ (called 
{\it orbifold charts}), 
where for each $i\in I$, $\tilde U_i$ is an open set in ${\Bbb R}^n$, for some $n$, $G_i$ is a finite 
group acting on $\tilde U_i$ and $p_i:\tilde U_i\to U_i$ is the quotient map, via an identification of 
$\tilde U_i/G_i$ with $U_i$ by some homeomorphism. 

\medskip

\centerline{\includegraphics[height=7cm,width=4.5cm,keepaspectratio]{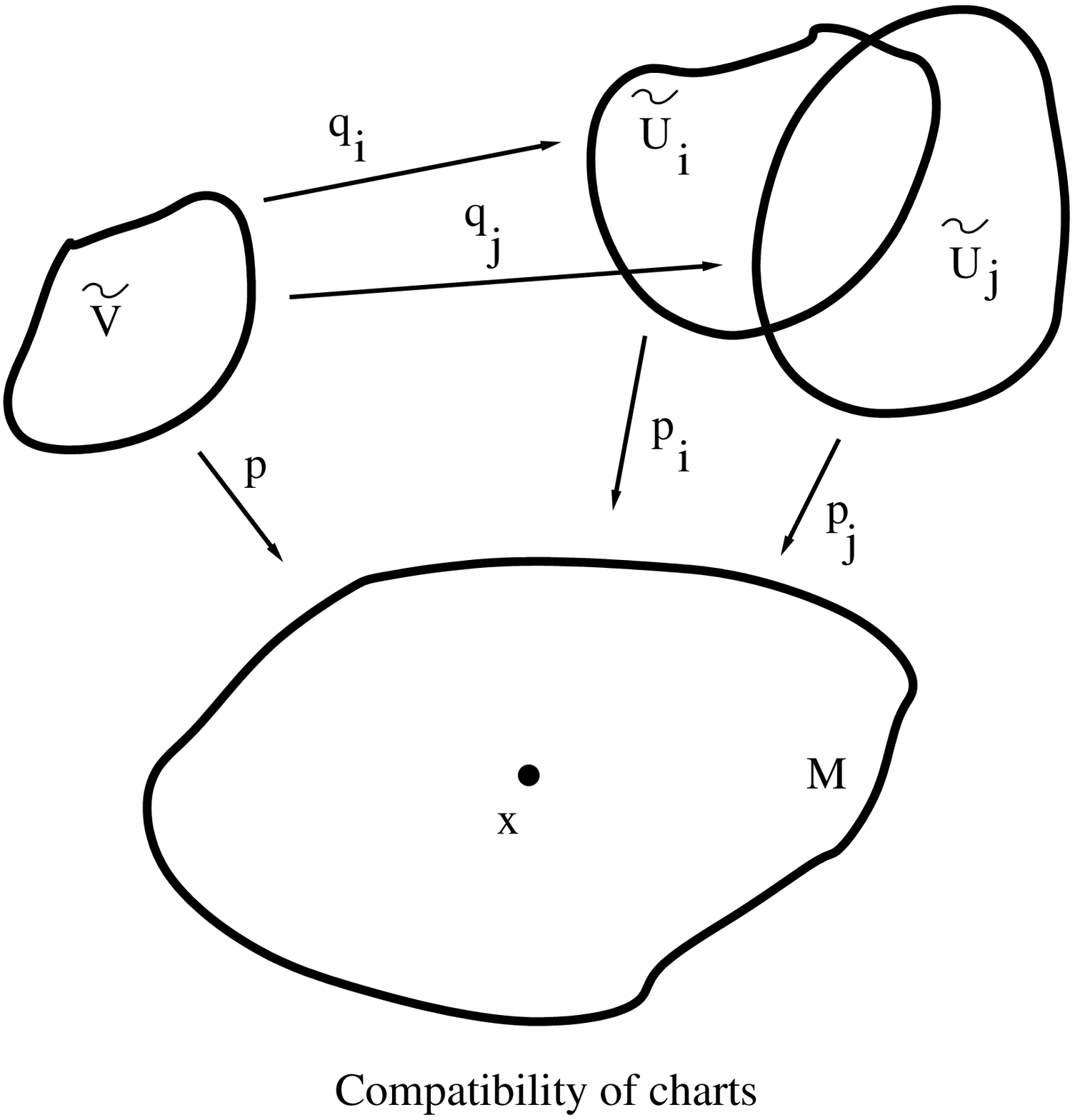}}

\centerline{Figure 1: Orbifold charts}

\medskip

Furthermore, the 
following compatibility condition is satisfied. Given $x\in U_i\cap U_j$, for $i,j\in I$ there 
is a neighborhood $V\subset U_i\cap U_j$ of $x$ and a chart $(\tilde V, p, G)\in {\cal O}_M$ together with embeddings 
$q_i:\tilde V\to \tilde U_i$ and $q_j:\tilde V\to \tilde U_j$ such that, $p_i\circ q_i=p$ and 
$p_j\circ q_j=p$ are satisfied.

The pair $(M, {\cal O}_M)$ is called an {\it orbifold} and $M$ is called its 
{\it underlying space}. 
 The finite groups $\{G_i\}_{i\in I}$ are called 
 the {\it local groups}. The union of the images (under $p_i$) of
 the fixed point sets of the action of $G_i$ 
on $\tilde U_i$, when $G_i$ acts non-trivially, is called the 
{\it singular set}. Points outside the singular set are called {\it regular points}. 
In case the local group is cyclic (of order $k$) acting by 
rotation about the origin on the Euclidean space, the image of the 
origin is called a {\it cone point} of {\it order} $k$. If the local group at some 
point acts trivially then it is a regular point. A regular point is also
called a {\it manifold point}.}\end{defn} 

Clearly, if all the local groups are either trivial or acts trivially
then the orbifold is a manifold.
The easiest  
example of an orbifold is the quotient of a manifold by a finite group. 
Also see Example \ref{example-covering} below.

In the rest of the paper we will not use the full notation of an orbifold as
defined above, unless it is explicitly needed. In general, we will use the same
letter $M$ both for the orbifold and its 
underlying space, which will be clear from the context.

As in the case of a manifold, {\it dimension} of a connected orbifold is defined. 
One can also define an {\it orbifold with boundary} in the same way we define a manifold with boundary. 
In the definition we have to replace ${\Bbb R}^n$ by the upper half space 
${\Bbb R}^n_+$. 

In this paper we consider orbifolds with boundary (may be empty).

A boundary component of the underlying space of an orbifold 
has two types, 
one which we call {\it manifold boundary} and the other {\it orbifold boundary}. These 
are respectively defined as points on the boundary (of the underlying space) 
with local group acting trivially or non-trivially. 

The notion of {\it orbifold covering space} was  
defined in \cite{Thu}. We refer the reader to this source for the basic 
materials and examples. But, we recall below the definition and properties we need.

\begin{defn}\label{covering}{\rm A connected orbifold $(\hat M, {\cal O}_{\hat M})$ is called an {\it orbifold covering space} 
of a connected orbifold $(M, {\cal O}_M)$ if there is a surjective map $f:\hat M\to M$, called 
an {\it orbifold covering map}, such that given $x\in M$ there is an orbifold chart 
$(\tilde U, p, G)$ with $x\in U$ and the following is satisfied. 

$\bullet$ $f^{-1}(U)=\cup_{i\in I} V_i$, where for each $i\in I$, $(\tilde V_i, q_i, H_i)\in {\cal O}_{\hat M}$, $V_i$ is a 
component of $f^{-1}(U)$, there is an injective homomorphism $\rho_i:H_i\to G$ and $V_i$ is 
homeomorphic to $\tilde U/(\rho_i(H_i))$ making the two squares in the following diagram commutative. 

\centerline{
\xymatrix{
\tilde U/(\rho_i(H_i))\ar[r]\ar[d]& V_i\ar[d]^{f|_{V_i}}&\ar[l]^{q_i}\tilde {V_i}\ar[d]^{\tilde f}\\
\tilde U/G\ar[r] & U&\ar[l]^p\tilde U.}} 

Where $\tilde f$ is $\rho_i$-equivariant.}\end{defn}

Here note that the map on the underlying spaces of an orbifold covering map need not be a
covering map in the ordinary sense.

\begin{exm}\label{example-covering}{\rm Given a group $\Gamma$ and a properly discontinuous 
action of $\Gamma$ on a manifold $M$, the quotient space $M/\Gamma$ has an 
orbifold structure and the quotient map $M\to M/\Gamma$ 
is an orbifold covering map. See [\cite{Thu}, Proposition 5.2.6].
Furthermore, if $H$ is a subgroup of $\Gamma$, then 
the map $M/H\to M/\Gamma$ is an orbifold covering map. Therefore, the quotient map of a finite 
group action on an orbifold is always an orbifold covering map.}\end{exm}

In general, an orbifold need not have a manifold as an orbifold covering space.

\begin{defn} (\cite{Thu}) {\rm If an orbifold has an orbifold covering space which is a manifold, then the 
  orbifold is called {\it good} or {\it developable}.}\end{defn}

In the above example $M/\Gamma$ is a good orbifold.

\begin{rem}\label{good-orbifold}{\rm One can show that a good compact   
$2$-dimensional orbifold has a finite sheeted orbifold covering space, which is a 
manifold ([\cite{Sc}, Theorem 2.5]). 
In the case of closed (that is, compact and the underlying space has empty boundary)
$2$-dimensional orbifolds,  
only the sphere with one cone point and the sphere with 
two cone points of different orders are not good orbifolds. See the figure 
below. Also, see [\cite{Thu}, Theorem 5.5.3].}\end{rem}

\centerline{\includegraphics[height=9cm,width=6.5cm,keepaspectratio]{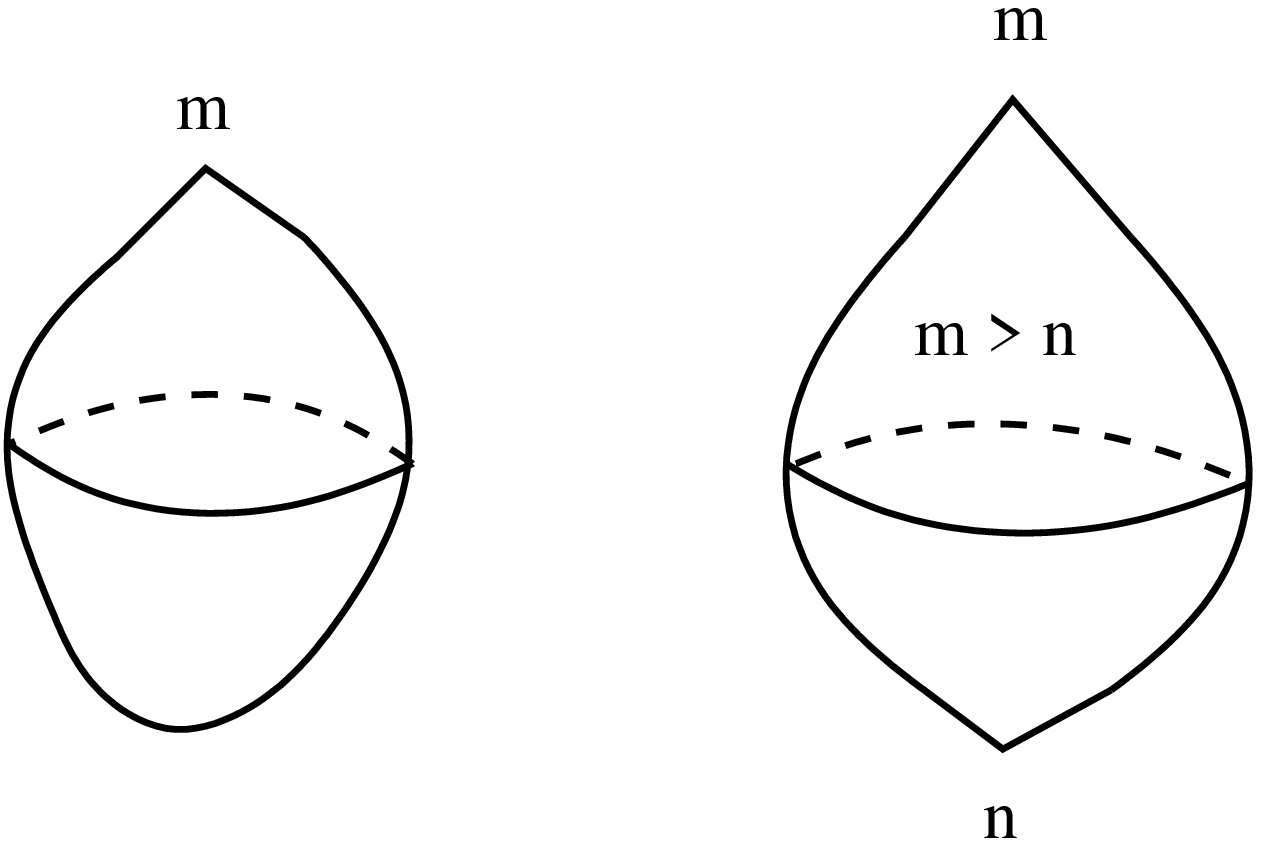}}

\centerline{Figure 2: Bad orbifolds}

\medskip

\begin{defn}{\rm Let $f:\hat M\to M$ be an orbifold covering map of 
connected orbifolds. Let $x\in M$ be a regular point and 
$\hat x\in \hat M$ with $f(\hat x)=x$. Then $\hat M$ is called the 
{\it universal orbifold cover} of $M$ if given any other connected orbifold covering 
$g:N\to M$ with $g(y)=x$, for $y\in N$, there is a unique $h:\hat M\to N$ so that $h(\hat x)=y$, $h$ 
is an orbifold covering map and 
$h\circ g=f$.}\end{defn}

\begin{defn}{\rm 
The universal orbifold cover always exists for a connected orbifold $M$ ([\cite{Thu}, Proposition 5.3.3]) and 
its group of Deck transformations  
is defined as the {\it orbifold fundamental group} of $M$. It is denoted 
by $\pi_1^{orb}(M)$.}\end{defn}

In Remark \ref{good-orbifold} we have already seen a classification
of closed good $2$-dimensional orbifolds. But in the literature I did not find any condition
which will say when an orbifold (compact or not) is good.

In this article we use the above classification of good closed orbifolds, to
prove the following theorem which gives a sufficient condition for
an arbitrary $2$-dimensional orbifold to be good.

\begin{thm} \label{mt} Let $S$ be a connected $2$-dimensional orbifold with finitely 
generated and infinite orbifold fundamental group. Then, $S$ is good.\end{thm}

\begin{rem}{\rm There are $2$-dimensional good orbifolds with 
    finite orbifold fundamental groups. The simplest example is the $2$-dimensional
    disc quotiented out by the action of a finite cyclic group acting by rotation
    around the center. See Remark \ref{final-remark} for a classification of
    such orbifolds.}\end{rem}

\section{Proof of Theorem \ref{mt}}

For the proof of Theorem \ref{mt} we need the following two lemmas, which follow  
from standard techniques and covering space theory of orbifolds.

\begin{lemma}\label{induced-covering} Let $M$ be a connected orbifold and
  $f:\tilde M\to M$ be a connected covering space of the underlying space of $M$.
  Then, $\tilde M$ has an orbifold structure induced from $M$ by $f$, so that
  $f$ is an orbifold covering map.\end{lemma}

\begin{proof} Choose an open covering of $\tilde M$ which consists of inverse
  images under $f$ of evenly covered open subsets of $M$ which are associated to
  orbifold charts of $M$, and then pull back the
  finite group actions using $f$. That gives the required orbifold structure on
$\tilde M$.\end{proof}

\begin{lemma}\label{injection} Let $q:\tilde M\to M$ be a finite sheeted 
orbifold covering map between two connected 
orbifolds. Then, the induced map $q_*:\pi_1^{orb}(\tilde M)\to 
\pi_1^{orb}(M)$ is an injection, and the image 
$q_*(\pi_1^{orb}(\tilde M))$ is a finite index subgroup of $\pi_1^{orb}(M)$.
\end{lemma}

\begin{proof} See [\cite{Che1}, Corollary 2.4.5].\end{proof}

\begin{proof}[Proof of Theorem \ref{mt}] Let $S$ be a $2$-dimensional orbifold as in the statement. 

Note that, the underlying space of a $2$-dimensional 
orbifold is a $2$-dimensional (smooth) manifold. (This follows easily from the discussion below
on the possible types of singular points on a $2$-dimensional
orbifold and the fact that manifolds of dimension $\leq 3$ has unique smooth structure.
Also see [\cite{Sc}, p. 422, last paragraph]). Therefore, 
after going to a two sheeted cover of the underlying space of $S$
and using Lemma \ref{induced-covering}, we can assume that the underlying space of $S$ is 
an orientable $2$-dimensional manifold. (In such a situation, it is standard to call the orbifold 
{\it orientable}).

For the remaining part of the proof we refer the reader to the discussion on pages
422-424 of \cite{Sc}.

It is known that, $S$ has three types of singularities: cone points, reflector lines and 
corner reflectors ([\cite{Sc}, p. 422]). 

Note that, other than the cone points the remaining singular set contributes to 
the orbifold boundary components 
of the orbifold. We take double of $S$ along the orbifold boundary, then 
we get an orientable orbifold $\tilde S$ which is a $2$-sheeted  
orbifold covering of $S$ and $\tilde S$ has only cone singularities ([\cite{Sc}, p. 423]). 
Furthermore, 
since $\pi_1^{orb}(\tilde S)$ is a finite index subgroup of $\pi_1^{orb}(S)$ 
(Lemma \ref{injection}), $\pi_1^{orb}(\tilde S)$ is again infinite and finitely generated. 

So, to prove the Theorem we only have to show that $\tilde S$ is a good orbifold. 

Note that, $\tilde S$ has no orbifold boundary component. If there is any manifold
boundary of $\tilde S$, then we take the double of $\tilde S$ along these boundary  
components and denote it by $D\tilde S$. Therefore, $D\tilde S$ is an orientable orbifold with
no boundary in its underlying space and only has cone singularities.
Also, if $D\tilde S$ is good then
so is $\tilde S$.

We now prove that $D\tilde S$ is good.

First, we consider the case when $D\tilde S$ is noncompact. 
The cone points 
form a discrete subset of $D\tilde S$ and hence we can write the underlying space of $D\tilde S$ as an 
increasing union of compact orientable sub-manifolds $S_i$, $i\in {\Bbb N}$ with no singular points 
on the boundary of $S_i$. (This can be done by taking a proper smooth map 
from the underlying space of $D\tilde S$ to ${\Bbb R}$.) Let $DS_i$ be the double of $S_i$. 
$DS_i$ is an orbifold with twice as many cone points as $S_i$. Suppose $DS_i$ is of 
genus $g_i$ and has $k_i$ cone points with orders $p_1, p_2,\ldots, p_{k_i}$. Then, $DS_i$ 
is a closed orientable orbifold with only cone singularities, and by
[\cite{Sc}, p. 424] the orbifold fundamental group
of $DS_i$ is given by the following presentation. 

$$\pi_1^{orb}(DS_i)\simeq\l a_1,b_1,\ldots,a_{g_i},b_{g_i},x_1,\dots,x_{k_i}\ |\ x_j^{p_j}=1,\ j=1,2,\ldots, k_i,$$
$$\Pi_{j=1}^{g_i}[a_j,b_j]x_1x_2\cdots x_{k_i}=1\r.$$

Hence, the abelianization of $\pi_1^{orb}(DS_i)$ is
isomorphic to ${\Bbb Z}^{2g_i}\oplus K_{k_i-1}$. Where, $K_{k_i-1}$ is a finite abelian
group with $k_i-1$ number of generators. 
Furthermore, $g_1\leq g_2\leq\cdots$ and $k_1\leq k_2\cdots$.

This shows that, since $\pi_1^{orb}(D\tilde S)$ is finitely generated (as $\pi_1^{orb}(\tilde S)$ is
finitely generated), there is an 
$i_0$ such that $g_j=g_l$ and $k_j=k_l$ for all $j,l\geq i_0$.
   
Therefore, $D\tilde S$ has finitely many cone points, all contained in $S_{i_0}$,
and outside $S_{i_0}$, $D\tilde S$ is 
a finite union of components, each homeomorphic to the infinite
cylinder ${\Bbb S}^1\times (0, \infty)$.
Next, we cut the infinite ends of
$D\tilde S$ at some
finite stage and denote the resulting compact orbifold again by $S_{i_0}$.

Note that, here $\pi_1^{orb}(S_{i_0})$ is a subgroup of 
$\pi_1^{orb}(DS_{i_0})$ and hence $\pi_1^{orb}(DS_{i_0})$ is infinite. 
Hence, from the classification of closed 
$2$-dimensional orbifolds using geometry (see [\cite{Thu}, Theorem 5.5.3]),  
it follows that $DS_{i_0}$ is a good orbifold. Since, $DS_{i_0}$ has even number 
of cone points, and in the case of two cone points they have the same 
orders. See Figure 2. Clearly, then $S_{i_0}$, and hence $D\tilde S$ (which is homeomorphic
to the interior of $S_{i_0}$) has an 
orbifold covering space, which is a manifold. 

Next, if $D\tilde S$ is compact then the same argument as in the above paragraph  
shows that it is good.

This completes the proof of the Theorem.
\end{proof}

\begin{rem}\label{final-remark}{\rm
We end the paper with a remark on $2$-dimensional orbifolds with finite
orbifold fundamental group. As in the proof of the Theorem, after going to a
finite sheeted covering and then doubling along manifold boundary components,
we can assume that the orbifold (say, $S$) is orientable and has finitely many 
cone singularities.
Ignoring the cone points we have a surjective
homomorphism $\pi_1^{orb}(S)\to \pi_1(S)$. Hence, the fundamental group
of the underlying space of $S$ is finite and therefore, the underlying space is
${\Bbb S}^2$ or ${\Bbb R}^2$. It is easy to see that in the
last case $S$ is good and in the
${\Bbb S}^2$ case we already have a classification (Remark \ref{good-orbifold}).}\end{rem}

\newpage
\bibliographystyle{plain}
\ifx\undefined\bysame
\newcommand{\bysame}{\leavevmode\hbox to3em{\hrulefill}\,}
\fi

\end{document}